\def\ZZ         {{\mathbb Z}}
\def\RR         {{\mathbb R}}
\def\CC         {{\mathbb C}}

\def\QQ         {{\mathbb Q}}
\def\PP         {{\mathbb P}}
\def\NN         {{\mathbb N}}
\def\ZZ         {{\mathbb Z}}

\def\A         {{\cal A}}
\def\B         {{\cal B}}

\def\I         {{\cal I}}

\def\P          {{\cal P}}

\def\S           {{\cal S}}

\def\Card      {{\rm Card}}

\def\Ker        {{\rm {Ker}}}

\def\dim        {{\rm dim}}

\def\Re         {{\mathfrak R}}

\def\deg        {{\rm deg}}

\def\cL{\mathcal{L}}

\def\cal        {\mathcal}

\documentclass{amsart}
\usepackage{amsmath}
\usepackage{amssymb}
\usepackage{verbatim}
\usepackage{eucal}
\usepackage{enumerate}
\usepackage{tikz}
\usetikzlibrary{decorations.markings,decorations.pathmorphing,cd}
\usetikzlibrary{arrows}
\usetikzlibrary{calc}
\usepackage{pgfplots}
\pgfplotsset{every axis/.append style={
                    axis x line=middle,    
                    axis y line=middle,    
                    axis line style={-,color=blue}, 
                    xlabel={$x$},          
                    ylabel={$y$},          
            }}

\usepackage{amsthm}
\usepackage{amssymb}
\usepackage{enumerate}
\usepackage{graphicx}
\usepackage{hyperref}
\newtheorem{thm}{Theorem}[section]
\newtheorem{theorem}{Theorem}[section]

\newtheorem{prop}[theorem]{Proposition}

\newtheorem{cor}[theorem]{Corollary}
\theoremstyle{definition}
\newtheorem{rem}[theorem]{Remark}
\newtheorem{dfn}[theorem]{Definition}

\newtheorem{exam}[theorem]{Example}

\theoremstyle{remark}
\newtheorem{remark}[theorem]{Remark}

\title{Braid monodromy and Alexander polynomials of real plane curves} 
\author{ Anatoly Libgober}
 
\address{Department of Mathematics\\
University of Illinois\\
Chicago, IL 60607}
\email{libgober@uic.edu}
\thanks{}
\begin{document}
\begin{abstract} We describe symmetries of the braid monodromy
  decomposition for a class of plane curves defined over reals
  including the real curves with no real points and proving new divisibility relations for Alexander
  invariants of such curves.
\end{abstract}
\maketitle
\centerline{\dedicatory{In Memory of John W.Wood.}}

\section{Preface and statement of results}\label{intro}

Alexander polynomial of a projective curve embedded into a smooth algebraic surface (cf. \cite{handbook})
is an invariant of the fundamental group of the complement to the
curve. It measures a degree of non-commutativity of the group 
and can be expressed in terms of geometric data of the surface and the curve,
including the local types and position of singularities of the curve on the
surface. An interesting problem is understanding which polynomials 
can appear as the Alexander polynomials of the fundamental groups
of the complements: this is a very special case of the fundamental problem 
of understanding the quasi-projective groups.

Divisibility
theorems give strong restrictions on the class of 
polynomials which can occur as the Alexander polynomials of projective
curves. One such result (cf.\cite{meduke},\cite{handbook}) asserts that Alexander polynomial 
of the fundamental group of the complement to an ample curve 
$C$ divides the product of the Alexander polynomials
of links of all singularities of the curve. In particular, if
such $C$ has ordinary nodes and cusps as the only singularities than the global Alexander
polynomial has a form $(t-1)^{a}(t^2-t+1)^b$. 
Focusing on the curves in a complex projective
plane rather than general smooth algebraic surfaces (as we will do in
this paper) one can easily see that  
 $a+1$ is the number of irreducible components of $C$
(cf. \cite{handbook}) but the range of multiplicities $b$ is far from 
clear.  In the case of curves in $\CC\PP^2$, the global Alexander polynomial also divides 
the Alexander polynomial of the link in the 3-sphere 
which is the boundary of a small tubular neighborhood
of a line with the link being defined as the intersection of the
curve with this 3-sphere. 
If this  ``line at infinity'' is transversal to the   
curve $C$ then the corresponding link
is the Hopf link with $d$ components where $d$ is the degree of $C$ and
 the Alexander polynomial of the latter is 
$(t-1)(t^d-1)^{d-2}$. One obtains that for curves with nodes and cusps 
the multiplicity of the factor $(t^2-t+1)$ is at most $d-2$ if $6
\vert d$ and is zero otherwise (for a similar
divisibility relation on surfaces more general than $\CC\PP^2$ see
\cite{handbook}).
\footnote{A different divisibility relations one obtains if
the line at infinity is selected to be non-transversal to the curve or contains
the singularities. This gives divisibility by the
Alexander polynomial of the complement to affine curve which is a
complement in $\CC\PP^2$ to the union of
the curve and the line. The Alexander polynomial of this affine curve may be different than the Alexander
polynomial of the projective curve. A slightly better than $d-2$ bound on the
multiplicity is given in \cite{crelle} Cor.3.13.}

This bound is much weaker than what so far was observed in examples. At the moment,
 it is unknown if the multiplicity of a primitive root of unity
of degree $6$ in the Alexander polynomial of a curve,  having as its singularities ordinary cusps and
nodes only,  has a bound independent of
$d$. The largest known multiplicity is $4$ for a curve of degree 12
with 39 cusps (cf. \cite{crelle}). It is 
known that if the Mordell Weil ranks of isotrivial elliptic
threefold (or isotrivial elliptic surfaces) are bounded, 
the multiplicities of the factors of the
Alexander polynomials of curves in this class, also are bounded independently of degree (cf. \cite{crelle} for
proof of both assertions). 

In this note, we discuss a new type of divisibility relations for the Alexander
polynomials for the complements to curves in $\CC\PP^2$ for which the defining equations have only
real coefficients \footnote{in the case of reducible curves, this
  allows for irreducible components not to have $\RR$  as the field of
  definition; we do not make any assumptions about the reality of the
  critical values of projections used to construct the braid monodromy
as was customary in previous works, cf. \cite{mina}}. The presence of
a real structure imposes restrictions on the
braid monodromy of the curve. The latter is an invariant of a curve $C
\subset \CC\PP^2$ 
and its projection onto a complex line $N \subset \CC\PP^2$, given as the homomorphism
$\pi_1(N\setminus Cr,b) \rightarrow B_d$. Here $b$ is a base point, $Cr$ is the subset of $N$
consisting of points over which the fiber of projection of $C$
has cardinality less than the degree $d$ of $C$, and $B_d$ is Artin's braid
group. If one views $B_d$ as the mapping class group of a disk with
boundary and $d$ marked points then the braid monodromy assigns to a
loop in $N \setminus Cr$ the class of
the diffeomorphism is given by the trivialization of the fibration of the
pair $(\CC\PP^2\setminus p,C)$, 
over this loop (here $p$ is the center of projection onto $N$, 
cf.  \cite{moishezon}, \cite{handbook} for a more
recent  exposition or section \ref{braidmonodromy} below).

The main results of the paper are the Theorem \ref{maindecomp}, its corollary given as
Proposition \ref{braidgarside}, the Theorem \ref{alexanderdivisibilitymain}, and its corollary the Proposition \ref{alexanderbound}.
  
 In Section  \ref{braidmonodromy}, we describe symmetry
 in the structure of this
homomorphism depending on the real structure of the curve. It appears that 
certain operators introduced by Garside 
(cf. \cite{garside}) play important role in the description of this 
symmetry and vice versa, the study of braid monodromy of real curves
gives geometric interpretation to some of Garside's identities. In
particular, if the projection is defined over $\RR$ and the
intersection of the finite set $Cr$ with the real locus of $N$ is empty
then the braid monodromy takes the class in $\pi_1(N \setminus Cr,b)$
(where $b$ is real) represented by the loop corresponding to $\RR\PP^1$ to the Garside word
$\Delta$. Braids corresponding to such loops were considered in \cite{hurwitz}
in the related context of Hurwitz
schemes. 
If the real part of the critical set $Cr \cap \RR\PP^1 \ne \emptyset$, then $\pi_1(N\setminus
Cr,b)$ contains three canonical loops: the one containing only critical
points in the real part of the critical set and two loops containing all
critical points in each of two connected components of $N\setminus
\RR\PP^1$. We describe constraints on corresponding braids and solve
the equations in the braid group to obtain an explicit form of the
braids corresponding to these canonical loops. 

In section \ref{presentations} we prove that the fundamental group  of
a curve over $\RR$ is a quotient of the fundamental group of a link
which is one of the closed braids attached to the curve in section \ref{braidmonodromy}. 
The argument here is purely topological and does not use an algebro-geometric
structure, and can be used in a different, for example symplectic, context.

In Section \ref{alexpolysection} we show that the global Alexander
polynomial of a curve over $\RR$ divides the Alexander polynomial of
 a link which is the closure of a braid associated with
the real structure and discussed in two previous sections. We calculate
these Alexander polynomials in some cases and make
the divisibility relations explicit. For
example for real curves without real points at all, the Alexander polynomial
divides  $(t^d-1)^{{{d-2} \over 2}}(t^{d\over 2}+1)(t-1)$. This gives 
${d\over 2}-1$ as a
new bound on the multiplicity of the factor $t^2-t+1$ in the Alexander
polynomial of a curve with nodes
and cusps over $\RR$ and no real points.  Moreover, we show that this bound is sharp at
least for such sextics. The last section also contains a discussion of 
braid monodromy of arrangements defined over reals, i.e. such that the equation of the union of
all lines is defined over $\RR$, but having only finitely many real
points. Such arrangements are perhaps of interest on their own. 
Note that the effect of complex conjugation on braid monodromy was
considered earlier in \cite{cohensuciu} (in connection with a study of
MacLane arrangements). 

Part of this work was done while the author participated in ``Braids''
program at ICERM in the spring of 2022. The author thanks
B.Guerville-Balle, A.Degtyarev, the anonymous referee for very useful comments and references in
connection with the earlier version
of this note, and most importantly Juan Gonzalez-Meneses for the answer to the author's question (cf. \cite{gm2023}
and Remark \ref{juananswer}).

\section{Complex conjugation and braid groups}\label{complconjsection}

Let $\P_N= \{P_1,\cdots P_N\}$ be an invariant under
conjugation subset in $\CC$. Let $P_0 \in \RR$. Complex conjugation induces on 
free group $\pi_1(\CC\setminus \P_N,P_0)$  an automorphism of order
2. In the following system of generators, this
automorphism has a particularly simple form. 

Recall that a good ordered
system of generators of $\pi_1(\CC\setminus \P_N,P_0)$ is given by
$N$ loops each consisting of a segment $I_i$ running from $P_0$ to the
vicinity of
one of the points $P_i$, followed by a counterclockwise loop running
along the boundary of a small circle centered at $P_i$ and  then
returning to $P_0$ along $I_i$. Moreover, it is assumed that these
loops are non-intersecting and ordered by the counterclockwise
ordering of their intersection points with a small circle centered at $P_0$.

Without loss of generality, we assume that the points in $\P_N$ are labeled so that $P_1,\cdots P_k \in
H^+$, $P_N,\cdots P_{N+1-k} \in H^-$, $P_i$ is the complex conjugate of
$P_{N+1-i}, i=1,\cdots, k$ \footnote{here $H^+$ (resp. $H^-$) denote the upper (resp. lower) half plane in $\CC$},  $P_{k+1},\cdots P_{N-k}$ form an increasing sequence of real numbers i.e. $P_{k+1}< \cdots  <P_{N-k}<P_0$. 
Let $\alpha_+: [0,1] \rightarrow H^+$ be an embedding of the interval $[0,1]$ such that $\alpha^+(0)=P_1,
\alpha^+(1)=P_k$ and $P_i \in \alpha^+[0,1], i=1,\cdots, k$ and the sequence $({\alpha^+})^{-1}(P_i) \in [0,1]$ is increasing. 
Let $I$ be the interval in $\CC$ which is the union of the interval $\alpha_+([0,1])$, followed by the line segment $P_k,P_{k+1}$, then followed by the part of the real axis running from $P_{k+1}$ to $P_{N-k}$, then followed by the segment $[P_{N-k}, P_{N+1-k}]$ and having as the 
final part the interval $\overline {\alpha_+([0,1])}$. Then $I \supset \P_N$ and the above order in $\P_N$ is induced by an orientation of $I$.
In the case when $N=2k$ and $\P_N\cap \RR=\emptyset$, $P_{k+1}=\bar P_k$ and in the above construction of $I$, instead of the segments with 
real endpoints we use the segment connecting $P_k$ and $P_{k+1}$ \footnote{the order of the points in $\P_N$ in $H^+$ can be selected arbitrarily but, as follows from this description, is consistent with 
the orientation of $\alpha_+([0,1])$ and the order of points in $H^+$
determines the order of the conjugate points in $H^-$.
In practice, it is convenient to order the points in $H^+$ according to the order of imaginary parts of the points.}. 

We select a good ordered system of generators $x_1,\cdots, x_{N-k}$ of the fundamental group
$\pi_1(H^{+}_{\epsilon}\setminus \bigcup_1^{N-k} P_i,P_0)$ of the complement to 
the set $P_1,\cdots P_{N-k}$ in a $\epsilon$-neighborhood
$H^+_{\epsilon}$ of the
closure of $H^+$ in $\CC\PP^1$ and extend it to a good ordered system
$x_1,\cdots, x_N$ adding loops, as the set each being the conjugate of a loop among the first $k$ loops in
the already selected system but using orientation and the order given by 
the above definition of a good ordered system.
  With these notations,
the involution on $\pi_1(\CC\setminus \P_N,P_0)$  induced by the complex conjugation $\gamma \rightarrow \bar \gamma$
of the oriented loops is given  by (writing from the right to the left):
\begin{equation}
   \bar x_i=x_{N+1-i}^{-1}, i=1,...k \ \ \ \bar x_i=x_{N-k}\cdots x_{i+1}
   x_i^{-1}x^{-1}_{i+1}\cdots x_{N-k}^{-1}, \ \  i=k+1,\cdots N-k
\end{equation}
Note that if at most one of $P_i, \  i=1,\cdots N$ is real, then the action is just
$\bar x_i =x_{N-i}^{-1}$. In particular, one has 
$$\overline{x_N\cdots x_1}=(x_N\cdots x_1)^{-1}$$

Let $Diff(D^2,\P_N)$ be the group of diffeomorphisms of a conjugation
invariant disk in $\CC$ containing $\P_N$ and
taking the set $\P_N$ into itself. Let $Diff^+(D^2,\P_N,\partial D^2)$ be its
subgroup consisting of diffeomorphisms which are orientation preserving and 
 constant on the boundary of the disk. The latter is a normal
subgroup of the former and the same is the case for the groups of connected
components of each of these groups.
The group $\pi_0(Diff^+(D^2,\P_N,\partial D^2)$ is Artin's braid
group $B_N$. The complex conjugation is an orientation-reversing element in
$Diff(D^2,\P_N)$ and conjugation by this element, acting as an inner automorphism of
$\pi_0(Diff(D^2,\P_N))$, acts as the outer automorphism of the normal subgroup $B_N=\pi_0(Diff^+(D^2,\P_N,\partial D^2)$.
We denote this automorphism as $\beta \rightarrow \bar \beta, \beta
\in B_N$. The group $B_N$ is a subgroup of the group of automorphisms
of  $\pi_1(\CC\setminus \P_N,P_0)$ and from the above definition,  one has 
\begin{equation}\label{conjugationdefinition}
    \bar \beta(x)=\overline{\beta (\bar x)} \ \ \ x \in
    \pi_1(\CC\setminus \P_N,P_0), \beta \in B_N
\end{equation}

It is immediate that (\ref{conjugationdefinition}) implies the following for the action of the conjugation on the 
standard generators $s_1,\cdots s_{N-1}$ of $B_N$ \footnote{these generators depend on a choice of $I$ or, more specifically, the choice of 
the part of $I$ containing the points in $H^+$ since the rest of interval $I$ described above canonically. Recall, that generators of 
the braid group, viewed as $\pi_1(Diff^+(D^2,\P_N,\partial D^2))$ are the half-twists corresponding to simple proper arcs connecting consecutive points 
of $\P_N$ (cf. \cite{fb}, 9.3.1). The choice of $I$ with orientation provides both the ordering and the arcs between consecutive points. The action
of the braid group on the free group in these generators is given by $s_i(x_i)=x^{-1}_ix_{i+1}x_i, s_i(x_{i+1})=x_i, s_i(x_j)=x_j,\forall  j\ne i,i+1$. In 
particular $s_i(x_N\cdots x_1)=x_N\cdots x_1$.}
, given by the orientation of $I$ from $P_1$ to $P_N$ i.e. the counterclockwise Dehn half-twists
corresponding to  the subintervals of $I$ connected consecutive points in $\P_N$, with 
the order along $I$:
\begin{equation}\label{unrealconj}
    \bar s_i=s_{N-i}^{-1} \ \ \ i=1,..,k-1,N-k+1,\cdots,N-1; 
\end{equation}
\begin{equation}\label{realconj}
\ \ \bar s_i =s_i^{-1} \ \
    i=k+1, \cdots N-k-1 
\end{equation}
Conjugates of generators in the remaining pair are given by \footnote{$\bar s_k$ is the clockwise Dehn twist about the segment 
$\overline {[P_k,P_{k+1}]}=[P_{N-k+1},P_k]$ and $\bar s_{N-k}$ is the clockwise Dehn twist about the segment $[P_{N-k},P_k]$}
\begin{equation}\label{specialconj}
\bar s_k=s_{N-k}\cdots
s_{k+1}s_k^{-1}s_{k+1}^{-1}\cdots s_{N-k}^{-1}
; \ \ 
\bar s_{N-k}=s_{N-k-1} \cdots s_{k+1}s_k^{-1}s^{-1}_{k+1}\cdots s^{-1}_{N-k-1}
\end{equation}

Conjugation on the braid group depends on the number of real points  in the set $\P_N$ 
and has particularly form (\ref{unrealconj}) if 
${\rm Card}(\P_N \cap \RR) \le 1$ or (\ref{realconj}) if $\P_N \subset \RR$ 
(the latter is the case considered in \cite{cohensuciu}).

Following Garside (cf. \cite{garside} sec. 1.2 and 2.1)
we will use the involution $\Re: s_i
\rightarrow s_{N-i}$ and the anti-homomorphism
$rev: B_N \rightarrow B_N$ which is rewriting a word in generators 
$s_i$ of the braid group or their 
inverses in reversed order. One has $rev(gh)=rev(h)rev(g), \forall
g,h\in B_N$.
We will use similar operations $\Re x_i=x_{N+1-i}$ and $rev$ on
generators $x_1,..,x_N$ of a free
group.  In particular, if ${\rm Card} \P_N \cap \RR \le 1$
(equivalently $k={N \over 2}$ or $k={{N-1} \over
  2}, N \ odd$) i.e. the action is given by (\ref{unrealconj}) then 
\begin{equation}\label{conjugationformulas}\bar s_i=\Re(s_i)^{-1} \ \ \ 
\bar x_i=\Re(x_i)^{-1} 
\end{equation}

It follows from \cite{garside} that with such restriction on $k$, 
the action of complex conjugation on Garside word satisfies $\bar \Delta=\Delta^{-1}$. Also, note that the complex conjugation 
and inner automorphisms generate $Aut B_N$ cf. \cite{dyer}: the
automorphism $\epsilon_n$ in that paper is the product of $\Re$ and
complex conjugation; in the case $\P_N \in \RR$, the automorphism $\epsilon_n$ is the complex
conjugation cf. 
(\ref{realconj}).



\section{Braid monodromy of curves over $\RR$}\label{braidmonodromy}


  Recall the definition of braid monodromy in the context of real
  curves.  Let $P \in \RR\PP^2$, $p_{\RR}:
  \RR\PP^2\setminus P \rightarrow N_{\RR}$ be the projection from $P$
  onto a line $N_{\RR}$ and  $\cL_{\RR}$ be the corresponding pencil of lines in $\RR\PP^2$.
 For each of these objects, the corresponding complexification will be denoted
  by the same letter but with $\RR$ changed to $\CC$. Complex conjugation
  acts on the set of  $\CC$-points of each of these sets, having 
the set of real points as the fixed point set. The fiber
of projection $p_{\CC}$ over $c \in N_{\CC}$ will be denoted $L_{\CC,c}$.
 
  Let $b \in N_{\RR}$ be a point selected so that
    $L_{\CC,b}$  is 
  transversal to the complexification $C_{\CC}$ of a  curve $C_{\RR}$. Let $Cr \subset N_{\CC}$
  be the subset of the points $c$ such that $\Card (p^{-1}(c) \cap C_{\CC})<d, d={\rm deg}(C)$.

Let $\gamma(t)$ be a loop in $N_{\CC}$ with initial and
endpoints being at $b \in N_{\RR}$ and situated in the upper half plane of $N_{\CC}$. Let $\bar
\gamma(t)=\overline{\gamma(t)}$ be its conjugate.  Consider
a trivialization of projection of the pair: $p: (p_{\CC}^{-1}(\gamma), C_{\CC} \cap p_{\CC}^{-1}(\gamma)) 
 \rightarrow \gamma$ i.e. a 
continuous map of pairs $\Phi: (I\times \CC, I \times [d]) \rightarrow
(p_{\CC}^{-1}(\gamma),C_{\CC} \cap p_{\CC}^{-1}(\gamma))$ (here
$I=[0,1]$ and $[d]$ is a fixed subset of $\CC$ of cardinality $d$) such that 
\begin{enumerate}
\item[$(i)$] $\Phi$ is compatible with projections of its source and
  target onto $I$ and $\gamma$
  respectively and in particular $\Phi(0,z)=\Phi(1,z) \in p_{\CC}^{-1}(b)$
  for any $z \in \CC$.
\item[$(ii)$] Restrictions of $\Phi$ onto $[0,1)\times \CC$ and $(0,1]
  \times \CC$ are homeomorphisms onto their targets. 
\item[$(iii)$] The trivialization is constant outside of a disk in $L_{\CC,b}$
  containing $L_{\CC,b}\cap C$ (in particular, for any $x \in L_{\CC,b}$ outside
  of this disk and $z\in \CC$ such that $\Phi(0,z)=x$ one has $\Phi(1,z)=x)$.
\end{enumerate} 
The monodromy along the loop $\gamma(t)$ is a diffeomorphism of the  pair
$(L_{\CC,b},C_{\CC}\cap L_{\CC,b})$
into itself sending $x \in L_{\CC,b}$ to $\Phi(1,z(x))$ where $z(x)$
is the solution to $\Phi(0,z(x))=x$. For a trivialization satisfying $(i),(ii),(iii)$, 
the braid corresponding to the isotopy class of such
diffeomorphism  
via identification of Artin's braid group with the mapping class
group of a disk with marked points will be denoted as $\beta(\gamma)$.

\begin{dfn}  {\it The braid monodromy of a plane curve is the homomorphism 
$\pi_1(N_{\CC}\setminus Cr,b) \rightarrow B_d$ which assigns to the
class of a
loop the braid in Artin's braid group corresponding to the
diffeomorphism given by the monodromy obtained from a trivialization
over the loop as described above.

Following \cite{moishezon} we present braid monodromy
as a factorization of the word $\Delta^2$ written as the product of
braids representing the value of braid monodromy on a sequence of a
good ordered system of generators of the fundamental group of the complement 
$\pi_1(N_{\CC}\setminus Cr,b)$.}
\end{dfn}

Complex conjugation acts on trivializations as follows. Clearly, since $C$ is
defined over $\RR$, for a loop $\gamma$ with a base point $b \in \RR$, 
one has $$(\overline{p_{\CC}^{-1}(\gamma),C_{\CC} \cap
  p_{\CC}^{-1}(\gamma)})=(p_{\CC}^{-1}(\bar \gamma),C_{\CC} \cap
  p_{\CC}^{-1}(\bar \gamma))$$ 

\begin{dfn}\label{conjtriv} {\it Conjugate of a trivialization $\Phi: (I\times \CC,I\times [d])
  \rightarrow (p^{-1}(\gamma),C_{\CC}\cap p_{\CC}^{-1}(\gamma))$ is the
  trivialization of $p_{\CC}$ over the loop $\bar \gamma$ given by
 \begin{equation}
   \bar \Phi (t,z)=\overline {\Phi (t,z)}
\end{equation} 
In particular, the monodromy diffeomorphism of $(p_{\CC}^{-1}(b),C
\cap p_{\CC}^{-1}(b))$ corresponding to trivialization $\bar \Phi$, in
terms of trivialization $\Phi$  
is given by: $x \rightarrow \overline{\Phi(1,\tilde z)}, x \in p_{\CC}^{-1}(b)$ where 
$\tilde z$ is determined by $\Phi(0,\tilde z)=\bar x \in p_{\CC}^{-1}(b)$}
\end{dfn}

\begin{remark} In general, it is impossible to trivialize over a loop
  the pair  $(p_{\CC}^{-1}(b),C
\cap p_{\CC}^{-1}(b))$ together with involution given by
conjugation. The type of involution (given by the number of fixed
points, i.e. the number of real points in the fiber) changes while one
moves along $\gamma$, but this procedure provides a well-defined 
diffeomorphism of pairs with involution.
\end{remark}

\begin{remark} An alternative way to define a braid monodromy is to
  use the so-called coefficient homomorphism  defined as the holomorphic map 
 assigning to an affine plane curve given by Weierstrass polynomial $y^d+\sum_{i=0}^{d-1} a_i(x)y^i=0$  the map $\CC
\rightarrow \CC^d$ given by $x \rightarrow
(a_{d-1}(x),\cdots,a_0(x))$. The restriction of this map to the complement to the set of critical values of projection of this curve
onto $x$-plane takes it to the space of the coefficients of
polynomials in one variable without multiple roots i.e. the
complement in $\CC^d$ to the discriminant hypersurface
$Discrim$. 
This complement is the base of a locally trivial fibration of the complement
in $\CC^{d+1}$ to the hypersurface given by equation
$y^d+a_{d-1}y^{n-1}+\cdots+a_0 \in \CC^{d+1}$ onto $\CC^d$ with
coordinates $(a_{d-1},\cdots, a_0)$. The action of the fundamental
group $\pi_1(\CC^d\setminus Discrim,b)$, which is isomorphic to Artin's braid group,
is induced by its action on the fundamental group of the fiber of this
fibration over the base point $b$. If $b \in \RR^d$ then the complex
conjugation acts on the braid group (since discriminant hypersurface 
is defined over $\RR$) but the specific form of this action on generators
depends on the choice of $b$.
This complex conjugation on $B_d$ is given by (\ref{unrealconj}) and
(\ref{realconj}) with $d=N$ and depends on the number of real roots of $y^d+a_{d-1}(b)y^{n-1}+\cdots+a_0(b)$. 
In particular, if the number of real roots is at most one, the action is
given by (\ref{conjugationformulas}) and if all roots are
real then one has $s_i \rightarrow s_i^{-1}$ for all $i$ (as in \cite{cohensuciu}).
\end{remark}

 Definition \ref{conjtriv} implies 
 the following relation between the braids that the braid monodromy
assigns to conjugate loops:

\begin{prop}\label{conjugatebraidRe} Let $\beta:
  \pi_1(N_{\CC}\setminus Cr) \rightarrow B_d$ be the homomorphism of braid monodromy
  of a plane curve over $\RR$ and let $\bar \gamma$ be 
 the complex conjugate of a loop $\gamma \in \pi_1(N_{\CC} \setminus Cr,b)$. 
Then 
$$\beta(\bar \gamma)=\overline{\beta(\gamma)}$$
and depending on ${\rm Card} (p^{-1}(b)\cap C_{\RR})$ the action of the 
complex conjugation on a factorization of $\beta(\gamma)$ is given by
(\ref{unrealconj}),(\ref{realconj}),(\ref{specialconj}).
In particular, if
 $p_{\CC}^{-1}(b)$ has at most one real point 
then 
   \begin{equation}\label{braidmonodromyconj} 
      \beta(\bar \gamma)=(\Re rev(\beta(\gamma)))^{-1}
    \end{equation}
 If all points of $p_{\CC}^{-1}(b)$ are real then one has:
\begin{equation}\label{braidmonodromyconjreal} 
      \beta(\bar \gamma)=(rev(\beta(\gamma)))^{-1}
    \end{equation}
i.e. coincides with the outer automorphism $\epsilon_d$ used in
\cite{dyer} \cite{cohensuciu}.
\end{prop}

\begin{proof} Indeed, the braid $\beta(\bar \gamma)$ being interpreted as an
  automorphism of the free group $\pi_1(p_{\CC}^{-1}(b)\setminus p_{\CC}^{-1}(b)\cap
  C,b)$ is the composition of 
conjugations, the automorphism corresponding to $\beta$ and the
conjugation i.e. is $\bar \beta$.  The equalities 
(\ref{braidmonodromyconj}) and (\ref{braidmonodromyconjreal})
follow from the identities 
(\ref{conjugationformulas}) (or (\ref{unrealconj})) and (\ref{realconj}).
\end{proof}

Next, we will find a conjugation invariant form of the braid monodromy
factorization 
of a real curve. Singularities of such a curve $C_{\CC}$ are either the points
with real coordinates or come in complex conjugate pairs. So are the
critical values of projections on the complex locus of a $\RR$- line:
they are real values at real critical points or come as complex
conjugate pairs. Recall that the image of $c \in C$ of projection on
$x$-axis $N$ from a point in $\RR\PP^2\subset \CC\PP^2$ is a critical
value if the line through the center of projection and $c$ intersects $C$ in
fewer than $\deg C$ points and a projection from a center $p\in \CC\PP^2$ is called
generic if each line in the pencil of lines through $p$ contains at most one singular or critical point on $C$ and each critical point in the smooth locus of $C$ is the simple tangency. We call a point in $\RR\PP^2$ {\it generic} if it is 
an arbitrary point in a Zariski dense subset of $\RR\PP^2$. 

The following Proposition describes which
critical points of generic projection are unavoidable on the real part of
$x$ axis (the target of the projection map).

\begin{prop}\label{genericprojections} Let $C$ be a projective plane curve over $\RR$ transversal to the line at infinity. 
 Let, as above, $p_P: \CC\PP^2\setminus P \rightarrow N_{\CC}$ be
  a projection from a point $P\in \RR\PP^2$ onto a fixed line $N$. 
 If $P$ is generic then the only critical points of $p_P$ on the real
 locus $\RR\PP^1 \subset N_{\CC}$ are either singular points of $C$ with 
 coordinates both being real or images of critical points of restriction of $p_P$ on
 the real locus $C_{\RR}$ of $C$. 
\end{prop}

\begin{proof} We will work in a coordinate system such that the projection from $P$ is just the 
projection $(x,y) \rightarrow x$ onto the $x$-axis.
First, notice that if a 
 coordinate system in $\RR\PP^2$  is generic then each singular point 
 either has both coordinates real or both coordinates have non-zero
 imaginary parts. We assume that $C$ is in a such coordinate system. 
 Then note also that the number of real lines through a point
  $(z,w,1) \in \CC^2 \subset \CC\PP^2$ is either infinite (if $(z,w)
  \in \RR^2)$ or is either 1 ($z,w$ are on a line $\CC=\RR^2$ defined
  over $\RR$) or zero. Consider the incidence correspondence 
  $\I \subset N_{\CC}\times \RR\PP^2$ 
 consisting of pair $(a,P)$ such 
  that $a$ is the image of a critical point of projection onto
  $N_{\CC}$ from $P$. The projection $\I \rightarrow \RR\PP^2$ is a
  finite cover and $\I$ is a real two-dimensional manifold. 
  Each real tangent transversal to the real locus of $C$ 
  either is a bitangent or contains singular points with one of the
  coordinates being not in $\RR$. There are no points in the latter class 
by genericity assumption and only a finite set of points in the former
class. Taking $P\in \RR\PP^2$ which pre-image in $\I$ has an empty intersection with this
finite set in $\I$ we get the required projection. 
\end{proof}

We select  coordinates in $\CC\PP^2$ so  
that the base point of the pencil is at infinity and the lines of the
pencils are the lines $x=c$ (i.e. the center of projection is
$(0,1,0$)) and we have the projection $p: \CC^2 \rightarrow \CC_x$ onto
the $x$ axis given by $y=0$. Moreover,  the trivialization of $p$ over
$\CC_x=\{y=0, z \ne 0 \}$ is given by projection $(x,y) \rightarrow y$
onto  $y$ axis $\CC_y$.
 
Let $Cr \subset \CC_x$ be the critical set of the projection and $N={\rm
  Card} Cr$. Let us view $Cr$ as a subset $\P_N$ discussed in Section \ref{complconjsection},
 select an
order in $Cr$,
a base point $b$, and a good ordered system of generators in $\pi_1(\CC
\setminus Cr,b)$ as described there.
We call this a complete good ordered system of generators {\it
  compatible 
with the real structure} and denote its elements\footnote{the order
of the loops in this system is from the left to the right but subscripts
specify the points in the set $Cr$ with the ordering described above.}
\begin{equation}\label{goodordered}
     \gamma_1, \cdots \gamma_h,\gamma^r_l, \cdots,\gamma^r_1,\bar \gamma^{-1}_h,....\bar \gamma^{-1}_1
\end{equation}
Here $\bar \gamma$ is the class of the loop containing the image of the conjugation map
$H^+\rightarrow H^-$ applied loop to $\gamma$. 

Finally, we will denote by $Dc\subset \CC_x$ be a  closed subset
bounded
by a loop with base point $b$ and such that $Dc \cap Cr$ is the set
of {\it real} critical values.

\begin{dfn}\label{braidsrealstructure} The factorization 
$$\beta_1\cdots \cdots \beta_N=\Delta^2, \ \ \ \ \  N=2h+l$$ 
where the braids $\beta_i$ are the images of the braid monodromy
homomorphism $\pi_1(\CC_x\setminus Cr,b) \rightarrow B_d$ 
corresponding to the loops (\ref{goodordered})
will be called {\it compatible with the real structure}.

The product of the braids corresponding to the loops $\gamma_1,\cdots,
\gamma_h$ will be denoted $\B_{H^+}$, the product of the braids
corresponding to $ \gamma^r_{l}, \cdots,\gamma^r_1$ we denote
$\B_{\RR}$, and the product of the
     braids corresponding to the remaining loops in this system we
     denote $\B_{H^{-}}$ so that 
$$\Delta^2=\B_{H^+}\B_{\RR}\B_{H^-}$$ 
\end{dfn}

\begin{thm}\label{maindecomp} Let $C$ be a projective plane curve
  over $\RR$. Let $p: \CC^2\rightarrow
  \CC=N_{\CC}$ be a projection of the affine part of $\CC\PP^2$ with the line at
  infinity being transversal to $C_{\CC}$. Let $b$ be the base point
  in the real locus of $N_{\CC}$  such that ${\rm Card}
  p^{-1}_{\CC}(b) \cap C_{\RR} \le 1$. 
The braid monodromy factorization corresponding to a  good ordered system of generators of
  $\pi_1(N_{\CC}\setminus Cr,b)$ compatible with real structure induces decomposition
\begin{equation}\label{braiddecomposition}
    \Delta^2= {\cal B}_{H^+} \cdot {\cal B}_{\RR} \cdot \Re(rev({\cal  B}_{H^+})) 
\end{equation}
where the braids $\B_{H^+},\B_{\RR}$ are as in Definition \ref{braidsrealstructure}.
\end{thm}

\begin{proof}  Since $\B_{H^-}=\bar \gamma_h^{-1}...\bar
\gamma_1^{-1}$, using Proposition \ref{conjugatebraidRe} this braid
can be written as:
$$(\Re (rev(\beta(\gamma_h)))^{-1})^{-1}  \cdots (\Re (rev(\beta(\gamma_1)))^{-1})^{-1}=
\Re(rev(\beta(\gamma_h)))....\Re(rev(\beta(\gamma_1)))$$
$$=\Re(rev(\beta(\gamma_h))rev(\beta(\gamma_{h-1}))\cdots
rev(\gamma_1))=
\Re(rev(\beta(\gamma_1)\cdots \beta(\gamma_h)))=$$ $$Re(rev(\B_{H^+}))$$
as claimed.
\end{proof}

\begin{rem}\label{juananswer} We would like to describe the braid $\B_{H^+}$ (and
  hence $\B_{H^-}$) in
  terms of the braid $\B_{\RR}$ corresponding to the real part of the critical
  locus i.e. to solve in the braid group the equation (\ref{braiddecomposition}).
 Unfortunately, as pointed out to the author by the referee, the constraint of the Theorem \ref{maindecomp} is not sufficient 
 for recovering the braids $\B_{H^{\pm}}$ and $\B_{\RR}$ even in the simples cases.
 Consider the case when $\B_{\RR}=1$, corresponding to the case of real curves without real points and for which the projection has no real critical values i.e. 
 \begin{equation}\label{maineqdelta2}
            \B \Re rev(\B) =\Delta^2
 \end{equation}
 Since $\Re (rev (\Delta))=\Delta$
  (cf. \cite{garside} Lemma 3), one expects that
 in decomposition (\ref{braiddecomposition}) $\B_{H^+}=\B_{H^-}=\Delta$:

However, in $B_4$, there is another solution to (\ref{braiddecomposition}) i.e. $\Gamma \Re rev \Gamma=\Delta^2$. 
\begin{equation}
  \Delta s^{-1}_1s_3 \Re(rev \Delta s_1^{-1}s_3)=\Delta s_1^{-1}s_3s_1s_3^{-1}\Delta=\Delta^2
\end{equation}
Note that  
\begin{equation}
\Delta s^{-1}_1s_3 =(s_1s_2s_3)^2 \ \ \ and \ \ \  s_3(s_1s_2s_3)^2s_3^{-1}=\Delta
\end{equation}
Nevertheless, we show below that for the algebraic curves in 
question the braids $\B_{H^{\pm}}$ are the a conjugates of "obvious" solutions. It would be interesting to find 
additional constraints on the braids $\B_{H^{\pm}},\B_{\RR}$ that allow to streamline their calculation and to understand the 
geometric significance of all solutions to (\ref{braiddecomposition}).
The equation (\ref{braiddecomposition}) has a certain similarity with the problem of finding roots in elements in the braid group
(cf. \cite{gm2003}) but a direct connection is not clear.

In response to the author's query, Juan Gonzalez-Meneses  informed the author that he obtained classification results
regarding the solutions of the equation (\ref{maineqdelta2}) for each of the Thurston-Nielsen classes: cf. \cite{gm2023}. 
\end{rem}

\section{Presentations of fundamental groups of real curves.}\label{presentations}

Recall that van Kampen's theorem (cf. \cite{vankampen},\cite{moishezon})
gives the following presentation in
terms of the braids $\beta_i$ described in definition
(\ref{braidsrealstructure}). 
\begin{equation}\label{vankampen} 
    \pi_1(\CC^2\setminus C)=\{x_1 \cdots x_d\vert \beta_i(x_j)=x_j\}
    \ \ \ \ j=1,\cdots d, i=1,\dots,\Card Cr
\end{equation}

In this section, we show that for curves over $\RR$, the fundamental
groups of the complement to $C_{\CC}$  are quotients of geometrically
defined and depending on real structure groups admitting van Kampen
type presentations but requiring fewer
relations than in (\ref{vankampen}).

With notations as in Section
  \ref{braidmonodromy}, let $Cr_{\RR} \subset Cr \subset N_{\CC}$ be
  the subset of {\it real} critical values of projection $p$. Recall that $Dc$
  is a disk in $N_{\CC}$ bounded by a loop based at $b$ and such that
  $Dc \cap Cr=Cr_{\RR}$.  

\begin{thm} \label{surjection0}
\begin{enumerate} 

\item  The group $\pi_1(p^{-1}(H^+\bigcup Dc) \setminus C,b)$ has a presentation:
\begin{equation} 
\pi_1(p^{-1}(H^+\bigcup Dc) \setminus C,b)=\{x_1,..,x_d \vert \beta_i(x_j)=x_j\}  
\ \ \ \ \ j=1, \cdots d, i=1,..,h+l 
\end{equation} \label{first}
\item Inclusion $p^{-1}(H^+\bigcup Dc)\setminus C  \rightarrow \CC^2\setminus C$ induces the surjection:
\begin{equation}
    \pi_1(p^{-1} (H^+\bigcup Dc \setminus C),b) \rightarrow
    \pi_1(\CC^2\setminus C,b) \rightarrow 1
\end{equation} \label{second}
\end{enumerate}       
\end{thm}

\begin{proof}  Recall that each loop in a good ordered system of generators
  bounds a disk containing a single critical value of projection
  $p$. The complement
  $\CC^2\setminus C$ can be retracted onto the union of
  $p$-pre-images of disks bounded by the loops in a good ordered system of
  generators (cf.  \cite{me2complexes}) and the fundamental group of
  the preimage of a disk corresponding to critical value $j$ has
  presentation $\{x_1,\cdots x_d \vert \beta_j(x_i)=x_i, i=1\cdots d \}$.
 Part  (\ref{first})
  follows from the van Kampen theorem about a union of spaces (cf. \cite{spanier}).

Part (\ref{second}) is a corollary of  part (\ref{first}) since the
set of relations of $\pi_1(\CC^2\setminus C,b)$ consists of the
same relations as the relations for $\pi_1(\pi^{-1} (H^+\bigcup Dc \setminus C),b)$ and
additional relations (corresponding to the critical points of
projection of $C$ in the lower half plane.
\end{proof}

\begin{prop}\label{presentations} Let $\beta_1,...,\beta_r \in B_d$ be a finite set of
  braids and $x_1,\cdots, x_d$ be a system of generators of a free
  group $F_d$.
 Let $\beta=\beta_1 \cdots \beta_r$ be
  their product, $G(\beta_1, \cdots \beta_r)$ (resp. $G(\beta)$)  be the
  quotients of the free group $F_d$ by the normal subgroup generated by elements
$\beta_i(x_j)x_j^{-1}, i=1,\cdots r, j=1,\cdots, d$ (resp. the relations
$\beta(x_j)x_j^{-1}, j=1,\cdots, d$). Then one has surjection:
\begin{equation}\label{surjection1}
      G(\beta) \rightarrow G(\beta_1,\cdots ,\beta_r) \rightarrow 1
\end{equation}
\end{prop}
\begin{proof} We shall show this by induction over $r$. Assume that for all  $1 \le j \le d$ the element
  $\beta_1\cdots \beta_{r-1}(x_j)x_j^{-1}\in F_d$ belongs
  to the normal subgroup $N_{r-1}$ of $F_d$ generated by $\beta_i(x_j)x_j^{-1}, 
  i=1,\dots , r-1$. Let $\beta'=\beta_1 \cdots \beta_{r-1}$. Then
  for any $x_j, j=1,\dots, d$
\begin{equation}\label{totalbraid}
 \beta(x_j)x_j^{-1}=\beta'(\beta_r(x_j))x_j^{-1}=\beta'(\beta_r(x_j))\beta_r(x_j)^{-1}
\beta_r(x_j)x_j^{-1}
\end{equation}
Let $\beta_r(x_j)=y_1y_2\cdots y_s$ where $y_k, k=1,\dots, s$ is one of
generators $x_1,\dots, x_d$ or their inverses, with possibly several of
$y_k$ corresponding to the same element among $x_j$. In particular, we
have $\beta'(y_k)y_k^{-1} \in N_{r-1}$ by the assumption of induction.
Then the right-hand side in (\ref{totalbraid}) can be written as:
\begin{equation}\label{totalbraids2}\beta'(y_1 \cdots y_s)y_s^{-1}\cdots y_1^{-1}\beta_r(x_j)x_j^{-1}=
\beta'(y_1)\cdots\beta'(y_s)y_s^{-1}\cdots y_1^{-1}
\beta_r(x_j)x_j^{-1} 
\end{equation}
The surjection $F_d \rightarrow F_d/N_{r-1}$ takes
$\beta'(y_s)y_s^{-1}$ to $1 \in F_d/N_{r-1}$ i.e. the last expression
in (\ref{totalbraids2}) goes to the same element as does
$\beta'(y_1)\cdots \beta'(y_{s-1})y_{s-1}^{-1}\cdots y_1^{-1}
\beta_r(x_j)x_j^{-1}$ and the latter goes to the same element as 
$\beta'(y_1) \cdots\beta'(y_{s-2})y_{s-2}^{-1}
\cdots y_1^{-1}
\beta_r(x_j)x_j^{-1}$ since $\beta'(y_{s-1})y^{-1}_{s-1}$ goes to $1
\in F_d/N_{r-1}$ and so on. In particular, the last expression in
(\ref{totalbraids2}) is the normal subgroup $N_r\subset F_d$ generated by $N_{r-1}$ and
the element $\beta_r(x_j)x_j^{-1}$ i.e. the subgroup of $F_d$
generated by the relations of $G(\beta_1, \cdots,\beta_r)$ which shows
the claim.
\end{proof}

\begin{remark} It is well known that the fundamental group of the
  complement to a singular curve is an invariant of
  equisingular isotopy of complex algebraic curves on surfaces (cf. \cite{handbook} for references therein). In the case of real curves a natural problem
  is to understand the {\it rigid} equisingular isotopy  classes 
   or at least the classes of equivariant (with respect to complex
  conjugation) equisingular isotopy (cf. \cite{rohlin} Sect. 4 for
  non-singular case). The fundamental group $\pi_1(\CC\PP^2\setminus
  C_{\CC},b), b \in \RR\PP^2$ endowed with the involution provides an
  invariant of classes of such restricted isotopy in the sense that
  for any two real curves $C_1,C_2$ isotopic via equivariant equisingular
  isotopy, there exist an isomorphism of the fundamental groups
  equivariant with respect to involutions induced by conjugation.
  For example, if $C$ is a pair of lines in $\CC^2$, then
  $\pi_1(\CC^2\setminus C,b, \ZZ)=\ZZ^2, i=1,2, b\in \RR^2$, and the involution
  induced by conjugation  has the matrix $\left(\begin{matrix} -1 & 0 \\
     0 & -1 \\ \end{matrix}\right)$ (resp. $\left(\begin{matrix} 0 & -1 \\
     -1 & 0 \\ \end{matrix}\right)$) if both lines are defined over
 $\RR$ (resp. both lines are imaginary) i.e. the additional structure
 distinguishes the classes of rigid isotopy.

For a curve over $\RR$, it follows from
Prop. \ref{conjugatebraidRe} that van Kampen presentation
(\ref{vankampen} for the braid monodromy in generators used in 
Proposition \ref{conjugatebraidRe} the automorphism of the free group
given by (\ref{unrealconj}) and (\ref{realconj})   
passes to an involution of the fundamental group of the
complement. Hence we obtained a calculation of this extra 
structure on the fundamental group. 

This involution can be encoded into an exact sequence:
\begin{equation}
  0 \rightarrow \pi_1(\CC^2\setminus C_{\CC},b) \rightarrow
  \pi^{\RR}_1(\CC^2\setminus C_{\CC},b) \rightarrow \ZZ_2 \rightarrow 0 
\end{equation}
in which the action of $\ZZ_2$ on $\pi_1(\CC^2\setminus C_{\CC},b)$ is
given by conjugation. The group in the middle 
 is a topological analog of Grothendieck's
fundamental group of a variety over $\RR$.
\end{remark}

\section{Alexander Polynomials}\label{alexpolysection}

Surjections of the previous section imply the divisibility relations
between the Alexander polynomials of the groups considered there.

Recall (cf. \cite{handbook} and the references therein)  that given a group $G$ and a surjection $\phi: G \rightarrow \ZZ$
one defines the Alexander polynomial as the order of the torsion part 
of the module over the ring of Laurent polynomials $\QQ[t,t^{-1}]$
with underlying $\QQ$- vector space being the quotient of $\Ker
\phi$ by its commutator 
with
constants extended to $\QQ$.  The module structure is defined by
requiring that the action 
of $t$ be  given  by the automorphism induced on $\Ker \phi$ by the conjugation by a lift to $G$ of the positive generator of the target $\ZZ$
of $\phi$  (this action when considered on abelianization of
$\Ker \phi$ is independent of a lift).

In terms of cyclic decomposition
\begin{equation}\Ker
\phi /(\Ker \phi)' \otimes \QQ=\oplus \QQ[t,t^{-1}]^a \oplus (\oplus_{i=1}^b
\QQ[t,t^{-1}]/(\Delta_i(t))  \ \ \ \Delta_i \vert \Delta_{i+1},
\end{equation}
this order is given by 
$$\Delta(t)=\Delta_1(t) \cdots \Delta_b(t)$$ 
if $a=0$ and $\Delta(t)=0$ if $a \ge 1$.

Recall also (cf. \cite{birman} Th. 2.2) that Artin showed that 
the fundamental group of the complement to a link in $S^3$
represented by a closed braid $\beta$ has the same presentation as the
group $G(\beta)$ from Proposition \ref{presentations} \footnote{in fact, \cite{birman}
shows the presentation of the fundamental group of link in $S^3$ is
the group $G(\beta)$ i.e. $G(\beta)=\{x_1,\cdots ,x_d \vert
\beta(x_i)x_i^{-1}, i=1,\cdots d\}$ with one relation deleted, but
there it is also pointed out that this relation is the combination of the remaining
$d-1$ relations.}

\begin{prop}\label{alexanderdivisibility} Let $G(\beta)$ and
  $G(\beta_1,\cdots,\beta_r)$ be associated
  with braids $\beta$ and $\beta_1, \cdots \beta_r$ groups
  with generators and relations described in Proposition
  \ref{presentations}. Let $\Delta_{\beta}$ and $\Delta_{\beta_1,\cdots
    ,\beta_r}$ be the Alexander polynomials of these groups relative
  to surjections $\phi_{\beta}$ and $\phi_{\beta_1,\cdots,\beta_r}$ onto $\ZZ$ which send each of their generators
  $x_1,\cdots ,x_d$ to the positive generator of $\ZZ$. Then 
$\Delta_{\beta_1,\cdots,\beta_r}$ divides $\Delta_{\beta}$. 
\end{prop}
\begin{proof}  The surjection (\ref{surjection1}) and the commutative
  diagram
$$\begin{matrix}G(\beta) & & \rightarrow & & G(\beta_1,\cdots \beta_r) \\
         & \searrow & & \swarrow & \\
     & & \ZZ & &
\end{matrix}
$$
induce the surjections of $\QQ[t,t^{-1}]$-modules:
\begin{equation}\label{surjection2}
 \Ker
\phi_{\beta} /(\Ker \phi_{\beta})' \otimes \QQ \rightarrow \Ker
\phi_{\beta_1,\cdots,\beta_r} /(\Ker \phi_{\beta_1,\cdots,\beta_r})' 
\otimes \QQ \rightarrow 0
\end{equation} 
Since the left group in (\ref{surjection1}) is the fundamental group
of the complement to a link in 3-sphere for which the Alexander module
is a torsion module, it follows from (\ref{surjection2}) that the
right group in (\ref{surjection1}) has as its Alexander module a
torsion module and the claim follows.
\end{proof}

\begin{thm}\label{alexanderdivisibilitymain}
 The Alexander polynomial of $\pi_1(\CC^2\setminus C)$
  divides the Alexander polynomials of the closed braids in 3-sphere
  associated with each of the braids $\B_{H^+}, \B_{\RR},
  {\B}_{H^+}\cdot {\B}_{\RR}$ in the braid group $B_d$. 
\end{thm}

\begin{proof} The theorem \ref{surjection0} identifies the group 
$\pi_1(p^{-1}(H^+\bigcup Dc) \setminus C,b)$ with the group
$G(\beta_1,\cdots \beta_{h+l})$ and Proposition \ref{alexanderdivisibility} shows
that the Alexander polynomial of this group divides the Alexander
polynomial of the closed braid $\B_{H^+}\B_{\RR}$. The surjection 
in Theorem \ref{surjection0} implies that the Alexander polynomial of
$\pi_1(\CC^2\setminus C)$ divides the Alexander polynomials of
$\pi_1(p^{-1}(H^+\bigcup Dc) \setminus C,b)$ as in the proof of 
Proposition \ref{alexanderdivisibility} which completes the proof.
\end{proof}

\section{Examples.}

In this section we discuss the braids $\B_{H^+},\B_{\RR}$
for the curves satisfying assumptions on the real part of the critical set and use them to get refined divisibility conditions for the
Alexander polynomials in corresponding  classes of plane singular
curves over $\RR$. 

The considered extreme cases are the case of the
arrangements of real lines and real arrangements of lines with zero-dimensional real locus,
the curves of even degree with empty real locus, and related curves of odd
degrees. To obtain a non-trivial divisibility relation one makes a
different selection of the braid $\B_{H^+}$ or $\B_{\RR}$ or their
product. In the case of arrangements of real lines, the braid
$\B_{H^+}$ is trivial, the Alexander polynomial of the closed braid in
$S^1 \times \CC$  is zero, and the divisibility relation is empty. On the
other hand, $\B_{\RR}=\Delta^2$, and one obtains a known divisibility
relation mentioned in Section \ref{intro}. Note that the Alexander module
of a link in $S^3$ (a closed braid) is torsion of the linking number
of any two components is non zero (cf. \cite{crowell}).

On the other hand, for the real curves with no real points, we show
that $\B_{H^+}=\Delta$ which leads to refined divisibility constraints
(cf. Prop \ref{alexanderbound}).

\begin{exam} We consider {\it maximally componentwise unreal
    arrangements} which are the arrangements over $\RR$ with the minimal number of real
points. Note the following:

\begin{prop} Let $\A$ be an arrangement over $\RR$. Then the set of its
  real points is not empty.
\end{prop}

\begin{proof}  This is immediate since the intersection point of a
  pair of conjugate lines is real.
\end{proof}

Let $\A_k$ be an arrangement over $\RR$  with $k<\infty$ real multiple
points. Such an arrangement has the
  form 
\begin{equation}\label{maxreal}
\prod_{j=1}^{r_1} (((x+ay)-n_1)^2+m^1_jy^2) \cdots \prod_{j=1}^{r_i}(((x+ay)-n_i)^2+m_j^iy^2)
\cdots \prod_{j=1}^{r_k} (((x+ay-n_k)^2+m_j^{r_k}y^2)
\end{equation} 
where  $$ 0 \ne a \in \RR, n_i \in \RR, \ \ \ n_1 >n_2 \cdots >n_k, \ \ \ m_i^j
\in \RR^+,  \pm m_{i'}^{j'} \ne \pm m_i^j \ \ \forall (i,j) \ne (i',j')$$
$\A_k$ contains $d=2\sum_1^k r_i$ lines with $k$ real multiple points $(n_i,0), i=1,..,k$  having respective 
multiplicities $2r_i, i=1,\cdots,k$.
The singular points, which are the intersection points of  lines with
different $n_i$'s for $a \ne 0$, have $x$-coordinates with non-zero imaginary parts
and come in conjugate pairs i.e. $Cr \A_k=\{ x \in \CC \vert x=n_i\}$. The only real points
of $\A_k$  are the points $(n_i,0),
i=1,\cdots, k$.

Let us describe the braid $\B_{\RR}$ in this case and more
specifically the braids corresponding to (the classes of the) loops
$\gamma_i^r \in \pi_1(\CC_x\setminus Cr,b)$
 (cf. Definition \ref{braidsrealstructure}) i.e. corresponding to
 paths running from $b$ along real axis to one of the points $(n_i,0)$ 
while circumventing points in $Cr$ following a small semi-circle in the
 upper half plane, then upon reaching the vicinity of $(n_i,0)$ following the full circle around $(n_i,0)$ and
 finally returning to $b$ along the same path.

\begin{prop} There is a collection of non-intersecting segments
  $\delta_1,\cdots \delta_i \cdots\delta_k$ in
  $L_{\CC,b}$ each containing a set $A_i$ or $2r_i$ points belonging
  to the $i$-th group of lines (\ref{maxreal}) in $\A_k$ 
such that braid $\B_{\RR}$ is a product of the conjugates of the full twists
$\Delta^2_{A_k}$ about $\delta_i$.
In particular $\Delta_i^2$ commute.
\end{prop}

\begin{proof}
Explicit form of the lines (\ref{maxreal}) shows that $y$
coordinates of the intersections of lines in this arrangement with the
fibers $L_{\CC,t}$ of projection used to calculate the braid monodromy
have the form 
\begin{equation} Re(y)=\lambda^i_j Im(y), \ \ \ \  i=1,\cdots, k,
  j=1,\cdots, r_i,  \lambda^i_j \ne \lambda^{i'}_{j'}, (i,j) \ne (i',j')
\end{equation}
and the braids corresponding to $\gamma_i^r$ can de described in terms
of the motion of
$d$ points $\A_k \cap \cal L_{\CC,b}$ along these lines
\footnote{Explicitly the intersection points of $L_{\CC,t}$ and the lines
  corresponding to the $i$-th factor in (\ref{maxreal}) are
  $y={{(n_i-t)(a \pm m_j\sqrt{-1})} \over {a^2+m_j^2}}$ and
  $\lambda^i_j=\pm{a \over m^i_j}$}.
Hence these $d$ points
are split into $k$ groups each containing 
$2r_i, i=1,...,k$ points. Each group moves toward the origin along the respective group of lines in $\CC$
 and arrives at $0 \in \CC$ at $k$ different moments. A group of $2r_i$
points while moving along their respective lines undergoing slight deviations at the
moments when other groups reach their critical points (corresponding
to $\gamma^r_i$ deviations from the real axis in $x$-plane),   
 Just before the time when the $i$-th group should arrive at the origin 
the points in this group undergo full twist (the braid corresponding to the
the singularity of $2r_i$ pairwise transversal lines) and returning to the 
the original position of the group along the same path. 

With each of $k$ critical points is associated ``vanishing segment''   
containing $2r_i$ points merging into $(n_i,0)$ which is
located at  $L_{\CC,b'}$ where $b'$ is point in the real part of $x$ axis
and where $\gamma_i^r$ starts the circle around $x=n_i$.  
Transporting this segment along the path of
$\gamma_i^r$ back from $b'$ to $b$ produces the vanishing segment of
this group in $L_{\CC,b}$. While $t$
moves from $b'$ toward $n_i$ and then completes the move along
semi-circle,  the set of $2r_i$ points merging in $(n_i,0)$ 
 move from half-plane $Re<0$ to half-plane $Re>0$. If $t$
continues to move in a negative direction instead of completing the full
circle around $(n_i,0)$, this segment will remain in the
right half plane when $t$ moves to the next critical value. Hence
vanishing segments corresponding to the critical values $n_i$ can be
selected inductively so that the segment corresponding to $n_i$ does not 
intersect the segments corresponding to $n_j, j<i$.
As a result, we obtain that the braids
corresponding to the loops $\gamma_i^r$ are full twists along the
collection of non-intersecting conjugation invariant segments in
$L_{\CC,b}$:
\begin{equation}\label{braidsmur}
     \Delta_{A_1}^2,
    \Delta_{A_2}^2\cdots\Delta_{A^k}^2
\end{equation}
where $\Delta_{A_i}$ is a rotation by 180 degrees of the group of points
$A_i$. In particular, these braids commute. The braid $\B_{\RR}$ is the
product of those twists. 
\end{proof}



\end{exam}

Next, consider other classes of algebraic curves.

\begin{prop}\label{braidgarside} Let $C$ be a real curve having degree $2d$ without real points and admitting a generic projection $\pi_{\CC}: C_{\CC}
  \rightarrow \CC$ having no critical values on the real axis.
Then the braid 
   $ {\B}_{H+}$ is a conjugate of the Garside element $\Delta_{2d}$ of the
   braid group $B_{2d}$. In particular, the braid monodromy of such curves satisfies the relation:
 \begin{equation}\label{braidmonodromyconstraint}
 \beta_1\cdots \beta_h=\gamma^{-1}\Delta_{2d}\gamma
 \end{equation}
 where $\beta_1,\cdots,\beta_h$ are the braids corresponding to the loops in a good ordered system of generators
 corresponding to the critical values in the upper half-plane and $\gamma \in B_{2d}$.
\end{prop}

\begin{rem} Consider the  following conditions on a curve and its projection: 

a) empty (i.e. without real points) over $\RR$ which 
the projection onto the $x$-axis is generic (as described before Prop.\ref{genericprojections}) and has no real critical values.

b) real curves given by an equation $f(x,y)=0$ where $f(x,y)$ has a constant sign for all $(x,y)\in\RR^2$ and projection onto $x$ axis is generic.

c) curves of an even degree admitting generic projection without real critical values.

Clearly a) implies b) and c) implies a). Indeed, a curve of an odd degree does have real points. If $c$ is a real critical value of projection of 
a curve as in c), then the critical point is either real, i.e. the curve is not empty, 
or a pair of conjugate critical points i.e. the projection is not generic. A curve of an even degree, admitting generic projection without
real critical values does not have to be empty (e.g. projection from an interior point of a real circle).

   
 \end{rem}
   
\begin{proof} {\it (of Prop. \ref{braidgarside})}.
Let $C_G \subset \CC\PP^2$ be a union of real smooth quadrics given by equation $G(x,y)=\prod_{j=1}^d (j(x^2+N^2)+y^2)=0,  \ (N\in \NN)$. Projection of $C_G$ onto $x$ axis has two critical values $x= \pm N\sqrt{-1}$ each corresponding to a singular point of $C_G$ (for $d>1$)
 each being a union $d$ smooth branches tangent to the fiber of the projection. The intersection index of each two branches at the critical point is 2.
 The alternative local model of this singularity (at the origin) is $\prod_1^d (j^2x-y^2)$. The subset of $j$-th branch of the latter germ over the loop 
 $\gamma(s)=e^{{2 \pi i s}}$ in the $x$-axis is $(e^{2 \pi i s}, je^{\pi i s}), j=1,\cdots d, 0 \le s \le 1$. It follows that the corresponding braid of this singularity is the half-twist $\Delta_{2d}$. Since, the loop 
$\gamma(s)$ (with the base point on the real axis) is homotopic to the loop represented by the real axis in $H^+\setminus i\sqrt{N}$, the braid of $C_G$ corresponding to the real axis is conjugate to $\Delta_{2d}$.

Let $F=0$ be the equation of $C$ and let $h$ be the number of critical points of projection of $C_{\CC}$ onto $x$-axis located in $H^+$.
Consider deformation $F_s=(1-s)F+sG, 0 \le s \le 1$ and the corresponding family of curves $C_{F_s}$. Critical points of projection of $C_{F_s}$ 
located in $H^+$ (resp. $H^-$) move to $N \sqrt{-1}$
(resp. $-N\sqrt{-1}$) along paths $\gamma_i(s)$ starting at $h$ critical points of $F$  formed by the critical points of $C_{F_{s,\CC}}, s<0<1$ and ending at the respective critical point of $G$. 

As long as the projection of 
the curve $C_{F_s,\CC}$ onto $x$ axis does not have real critical values, this curve remains transversal to $\RR \times \CC \subset \CC^2\subset \CC\PP^2$ where the first factor is the real locus of $x$-axis and the second one is $y$-axis. 
In particular, the braid $\B_{H^+}$ corresponding to the curve  $C_{F_s,\CC}$ does not change. For $s$ such that  $\gamma_i(s)$ crosses the real axis from $H^+$ to $H^-$, the conjugate $\bar Q$ of the corresponding critical point $Q$ of $C_{F_s,\CC}$ has the same local type as does $Q$. 
The critical value of projection of $\bar Q$ crosses from the opposite half-plane of the critical value at $Q$ and moves along $\overline {\gamma(s)}$. In particular, the factor of $\B_{H^+}$ 
which the braid monodromy associates with the element in the good ordered system of generators corresponding to the critical value on the path 
$\gamma_i(s)$ moving toward
the real axis is the braid corresponding to the same element in $\pi_1(H^+\setminus Crit(F(s))$ and the same singularity as before the 
crossing the real axis. 
So there will be no change in the braid $\B_{H^+}$ and the braid corresponding to $C_F$ coincides with a conjugate of the braid of $C_G$ i.e. $\gamma \Delta_{2d}\gamma^{-1}, \gamma \in B_{2d}$. 
\end{proof}


Recall that an acnode of multiplicity $l$ (cf. \cite{wall}) is a germ of a curve defined
over $\RR$ which set of real points consists of a single point and
the set of complex points is a union of $l$ transversal smooth branches. Note that this implies that 
 must be even i.e. $l=2k$ and the branches are pairwise conjugate.

\begin{prop}\label{acnode} Let $C$ be a curve of degree $2d$ defined over $\RR$ for which the set
  of real points consists of a single acnode of multiplicity $2k$.
 Then up, to conjugation in $B_{2d}$, one has $\B_{\RR}=\Delta_{2k}^2$ and $B_{H^+}=\gamma \Delta_{2d} \Delta_{2k}^{-1}\gamma^{-1}=\B_{H^-}$ where $\Delta_{2k}$ is a 180 degrees rotation of a complex conjugation invariant subset of $2d$ points.
\end{prop}

\begin{proof}  Let $F(x,y)$ be an equation of $C$. As in Prop. \ref{braidgarside} 
the assumption of Prop. \ref{acnode} implies that if there exists $(x_0,y_0)$ such that 
$F(x_0,y_0)>0$ then $F(x,y) \ge 0$ for all $(x,y) \in \RR^2$ with the vanishing taking place only at the acnode. 
As in Prop. \ref{braidgarside} 
consider the family of curves $C_s$ with the equation $F_s(x,y)=(1-s)F+sG=0$. Then $F_s(x,y) >0$ for all $(x,y),0<s \le 1$.
The local braid of acnode is (a conjugate in $B_{2d}$ of) the full twist $\Delta_{2k}^2$ on a set of $2k$ points. It coincides with the component $B_{\RR}(C_0)$ 
of the braid of the curve $C$ corresponding to the real part of the set of critical values. Moreover, when $s$ moves away from zero, 
the critical value $\kappa$ of the projection at the acnode splits for $0< s <<\epsilon$ into a finite set $\S(s)$ of $K$ conjugate pairs of points. We select the base point $b$ (cf. beginning of Section \ref{braidmonodromy}) 
in proximity to $\kappa$ so that Garside operations on elements of $B_{2k} \subset B_{2d}$ 
coincide with those induced from $B_{2d}$ i.e. $\Re_{B_{2d}} (rev_{B_{2d}} \Delta_{2k})=\Delta_{2k}$. 
A small loop $\gamma_{\epsilon}$ around $\kappa$, also containing the set $\S$ inside it, is taken by the braid monodromy to $\B_{\RR}(C_0)$ but 
the loop $\gamma_{\epsilon}$ is also split into a product of two conjugate loops $\gamma^+_{\epsilon}, \gamma_{\epsilon}^-$ each containing inside the subsets $\S \cap H^+$ and $\S \cap H^-$ respectively.  The local version of the argument in Prop.\ref{braidgarside} deforming the germ of the acnode to the special germ having only two conjugate critical values shows that the braid monodromy takes each of  $\gamma_{\epsilon}^{\pm}$ to the rotation by 180 degrees $\Delta_{k}$.
This implies that the loop $\B_{H^+}(C_0)\beta(\gamma_{\epsilon}^+)$ is a conjugate of $\Delta_{2d}$. Hence $\B_{H^+}$ 
is a conjugate of $\Delta_{2d}\Delta_{2k}^{-1}$.
\end{proof}

Now we turn to the Alexander polynomials and explicit divisibility relations.

\begin{prop}\label{alexanderbound} Let $C$ be a real curve of even degree $d$ admitting a projection with no
  real critical values.  Then
\begin{equation}
      \Delta_C(t) \vert (t^d-1)^{{d \over 2}-1}(t^{d\over 2}+1)(t-1)
\end{equation}
In particular, the multiplicity of the root $exp{{2\pi i }\over 6}$ of the Alexander polynomial of such 
a curve, having nodes and cusps as the only singularities, is at most ${d \over 2}-1$.
\end{prop}
\begin{proof} We need to find the Alexander polynomial of the link
  which is represented by the closed braid on $d$ strings given by $\Delta$ i.e. the
  rotation by 180 degrees. As was already mentioned, such a link also can be described as the link
  of  plane curve singularity given by local equation $\prod_{j=1}^{d \over
    2}(j^2x-y^2), $ (when preimage of $x=1$ are integers $\pm j, j \le
  {d \over 2}$) or equivalently $x^{d \over 2}-y^d=0$. The Alexander
  polynomial of such a link is the
  characteristic polynomial of its monodromy. This is a weighted homogeneous
  singularity with weights of $x$ and $y$ being ${d \over 2}$ and
  $d$. Now the claim follows from Brieskorn-Pham-Milnor classical
  calculations (cf. \cite{milnor} Sec. 9).
\end{proof}

For the curves with a single acnode we obtain as an immediate consequence of
Proposition \ref{acnode} since $\B_{\RR}=\Delta_k^2$:

\begin{cor} The Alexander polynomial of a curve over $\RR$ with a single
  real point which is an acnode divides the Alexander polynomial of the
  link which is the closure of the braid $\Delta_d\Delta^{-1}_k$.
\end{cor}

\begin{exam} The above Proposition gives a sharp ``estimate of  the
  degree of Alexander polynomials'' of real curves of degree 6 without
  real points.

  Recall that
  the Alexander polynomial of a plane curve of degree $6$ with $k>6$ cusps
  and at most nodes as other singularities is given by
  $(t^2-t+1)^{k-6}$. More specifically, for a curve of degree $d$ with
  cusps and nodes as the only singularities, denoting the 
  zero-dimensional subset of $\PP^2$ formed by cusps as $\Xi$, the Alexander
  polynomial is equal  to $1$ if $d$ is not divisible by $6$ and
  otherwise is equal to $(t^2-t+1)^s$ where
  $s=\dim H^1(\PP^2,\I_{\Xi}(d-3-{d \over 6}))$ with 
  $\I_{\Xi}$ denoting the ideal sheaf of functions vanishing at
  points of $\Xi$. For $d=6$ the latter is equal to 
$ H^0(\PP^2,\I_{\it  cusps}(2))-\chi(\PP^2,\I_{cusps}(2))=k-6$ 
  since $\dim
  H^0(\PP^2,\I_{\Xi}(2))=0$ if the number of cusps is greater than
  $6$ (if the cardinality of ${\Xi}$ is $6$ this dimension can be
  either $1$ or $0$ depending on 6 cusps being positioned on a conic
  or not).

The divisibility theorem of \cite{meduke} for a complex curve of degree $d$
transversal to the line at infinity tells that the Alexander
polynomial of the curve divides $(1-t)(1-t^d)^{d-2}$.
Indeed, the latter is the Alexander polynomial of the link at infinity which is the link of the closure of the
braid $\Delta^2$ (the Hopf link).  This bounds the exponent of
the Alexander polynomial of sextic by $4$. 
On the other hand by \ref{alexanderbound} the Alexander polynomial 
  of a sextic with no real points should divide
  $(t^3-1)(t+1)^2(t^2-t+1)^2(t-1)$. There is no sextic with 9 cusps with no
  real points since the number of cusps that are not real is even.
On the other hand, there exist a real sextic with 8 cusps and no real
points (cf. \cite{degtsextics}, \cite{itenberg}) for which 
the multiplicity of the factor $(t^2-t+1)$ is 2.
\end{exam} 

\begin{rem} It is not hard to calculate the Alexander polynomial of
  the link which is the closed braid corresponding to a curve of
  an odd degree having no real critical values. In this case the
  projection without critical points yields again rotation of a set
  with an odd number of points and the corresponding braid is again
  $\Delta$. Such link appears as the link of singularity $yx^{{d-1}
    \over 2}-y^{d+1}$ and calculation of the characteristic
polynomial of this weighted homogeneous singularity (with non-integer
weights, in which case one can use \cite{milnororlik}) one obtains
\begin{equation}
      (t^d-1)^{{d-1} \over 2}(t-1)
\end{equation}  
This may lead for example to a restriction on the degrees of the Alexander
polynomials with only ordinary triple points (in which case $d$ must
be divisible by $3$). It is unlikely however that this can give a sharp
bound on the degree.
\end{rem}

\end{document}